\documentclass[11pt]{amsart}
\usepackage[utf8]{inputenc}
\usepackage{amsmath}
\usepackage{graphicx}
\graphicspath{ {./images/} }
\usepackage{svg}
\usepackage{amsfonts}
\usepackage{mathtools}
\usepackage{amssymb}
\usepackage{float}
\usepackage{ragged2e}
\usepackage{extarrows}
\usepackage{dirtytalk}
\usepackage{accents}
\usepackage{upgreek}
\usepackage{cancel}
\usepackage{hyperref}
\usepackage{url}
\newcommand{\mypm}{\mathbin{\smash{%
\raisebox{0.35ex}{%
            $\underset{\raisebox{0.5ex}{$\smash -$}}{\smash+}$%
            }%
        }%
    }%
}
\DeclareMathOperator*{\esssup}{ess\,sup}
\DeclareMathOperator*{\essinf}{ess\,inf}
\usepackage{amsthm}
\usepackage[numbers]{natbib}
\usepackage{xcolor}
\newtheorem{theorem}{Theorem}[section]

\newtheorem{proposition}[theorem]{Proposition}
\newtheorem{example}[theorem]{Example}
\newtheorem{lemma}[theorem]{Lemma}
\newtheorem{definition}[theorem]{Definition}
\theoremstyle{remark}
\newtheorem*{remark}{\rm{\textbf{Remark}}}

\numberwithin{equation}{section}
\usepackage{pgfplots}
\definecolor{mygreen1}{RGB}{20, 180, 120}

\title[Integral means of solutions to the Robin boundary problem]{Integral means of solutions of the one-dimensional Poisson equation with Robin boundary conditions}
\author{Christos Papadimitriou}
\address{Aristotle University of Thessaloniki\\
         Department of Mathematics\\
         Thessaloniki, Greece, 546 35}
\email{papadimitc@math.auth.gr}
\date{December 2023}

\begin{document}
\maketitle
\begin{abstract}
Consider the one-dimentional Poisson equation \(-u''=f\) on the interval \([-\pi,\pi]\), where \(f\) is an non-negative integrable function, with Robin boundary conditions \(-u'(-\pi)+\alpha u(-\pi)=u'(\pi)+\alpha u(\pi)=0\), where \(\alpha>0\) is a constant. In this paper we prove inequalities for the convex integral means of solutions of this problem, using the polarization of functions and its properties. We find solutions with maximal convex integral means and prove their uniqueness.
\end{abstract}
\maketitle
\section{Introduction}
Imagine a metal rod or a wire of length \(2\pi\) that is heated through by an external heat source with distribution \(f\in L^1[-\pi,\pi]\). We allow the rod to interact with its environment by submerging its ends in liquid or gas with temperature zero and letting heat to be dissipated via Newton's law of cooling, i.e. the heat flux at the end-points is proportional to their temperature. The mathematical formulation of this law is achieved by introducing Robin boundary conditions to our setting. Therefore, we are interested in the temperature function \(u_f\) of the steady-state of the rod, meaning that the temperature of the rod is independent of the time variable.
\par Putting all this into equations, we have that \(u_f\) has to satisfy the Poisson equation
\begin{equation}\label{Poisson}
-u''=f\text{ on }[-\pi,\pi]
\end{equation}
with Robin boundary conditions
\begin{equation}\label{Robin}
-u'(-\pi)+\alpha u(-\pi)=u'(\pi)+\alpha u(\pi)=0,
\end{equation}
 for some constant \(\alpha>0\) (which might depend on the material of the rod and other parameters).
 \par 
Recently, J. J. Langford and P. McDonald \cite{Langford}, studying this problem, provided an explicit formula for the temperature \(u_f\) of the solution (see Proposition \(\ref{green}\)). They were also interested in the convex integral means and the \(L^p\)-norms of \(u_f\) and how these quantities compare with the corresponding ones of \(u_{f^\#}\), where \(f^\#\) is the symmetric decreasing rearrangement of \(f\) (see Definition \ref{sdr}). To this end, they proved a comparison principle (\cite[Theorem 1.3]{Langford}) considering \(u_f\) and \(u_{f^\#}\).
Related research has been contacted considering other extremal problems in our setting. In \cite{Betsakos}, for example, D. Betsakos and A. Solynin studied the temperature gap or oscillation 
\begin{equation*}
    \operatorname{osc} (u_f)=\max_{[-\pi,\pi]}u_f(x)-\min_{[-\pi,\pi]}u_f(x)
\end{equation*}
of \(u_f\). Already in \cite{Langford}, it was observed that among all rearrangements of the heat source \(f\) (see Subsection 2.2 for more on rearrangements), this quantity is not maximized by \(u_{f^\#}\). D. Betsakos and A. Solynin \cite{Betsakos} determined the exact rearrangement of \(f\) that yields the maximal temperature gap for a special class of heat sources.
In the present paper, we investigate how the convex integral means and the \(L^p\)-norms of \(u_f\) are affected under rearrangements of the heat source \(f\) such as the s.d.r \(f^\#\) but also its polarizations \(f_H\) (see Definition \ref{polar}). Let us now present the main results of this paper.
\begin{theorem}[Polarization and convex integral means]\label{polarconvex}
Let \(\alpha>0\), \(0\leq f\in L^1[-\pi,\pi]\) and \(u_f\) be the solution to the Robin problem \((\ref{Poisson})\), \((\ref{Robin})\). Let \(b\in(-\pi,0)\cup(0,\pi)\). Let \(f_H\) be the polarization of \(f\) w.r.t. \(b\), and let \(u_{f_H}\) be the solution to the Robin problem with heat source \(f_H\). Then
\begin{equation}\label{ineq polar}
\int_{-\pi}^\pi \phi\bigl(u_f(x)\bigr)dx\leq\int_{-\pi}^\pi \phi\bigl(u_{f_H}(x)\bigr)dx
\end{equation}
for any convex and increasing function \(\phi:\mathbb{R}\to\mathbb{R}.\)\\
Furthermore, if \(\phi\) is strictly increasing, then \((\ref{ineq polar})\) holds as an equality if and only if \(f=f_H\) a.e. on \([-\pi,\pi]\).
\end{theorem}
\vspace{3mm}
In particular, for \(\phi(\cdot)=|\cdot|^p,\) where \(p\geq1\), we get inequalities for the \(L^p\)-norms of \(u_f\) and \(u_{f_H}\). Letting \(p\to+\infty\) we get the inequality for their \(L^\infty\)-norms, i.e.
\begin{equation}\label{Lp polar}
||u_f||_{L^p}\leq ||u_{f_H}||_{L^p} \text{, for } 1\leq p\leq+\infty .
\end{equation}
\begin{theorem}[S.d.r. and convex integral means]\label{sdrconvex}
Let \(\alpha>0\), \(0\leq f\in L^1[-\pi,\pi]\) and \(u_f\) be the solution to the Robin problem \((\ref{Poisson})\), \((\ref{Robin})\). Let \(f^\#\) be the symmetric decreasing rearrangement of \(f\) and \(u_{f^\#}\) the solution to the Robin problem with heat source the function \(f^\#\), then 
\begin{equation}\label{ineq sdr}
\int_{-\pi}^\pi \phi\bigl(u_f(x)\bigr)dx\leq\int_{-\pi}^\pi \phi\bigl(u_{f^\#}(x)\bigr)dx
\end{equation}
for any convex and increasing function \(\phi:\mathbb{R}\to\mathbb{R}.\)\\
If \(\phi\) is strictly increasing, then \((\ref{ineq sdr})\) holds as an equality if and only if \(f=f^\#\) a.e. on \([-\pi,\pi]\).
\end{theorem}
Again, \((\ref{ineq sdr})\) implies that
\begin{equation}\label{Lp sdr}
||u_f||_{L^p}\leq ||u_{f^\#}||_{L^p}\text{, for }1\leq p\leq+\infty.
\end{equation}
 Note that the inequality \((\ref{ineq sdr})\) is the previously mentioned comparison principle of J. J. Langford and P. McDonald. They achieved this result using Baernstein's \(\star\)-function. They did not investigate the existence of conditions for when the equality is attained.
 \par
Using our method, we get the precise condition on when \((\ref{ineq sdr})\) is attained as an equality and we gain the intermediate Theorem \ref{polarconvex}, and, therefore, we can interpolate the inequality \((\ref{ineq sdr})\) as so
\[\int_{-\pi}^\pi \phi\bigl(u_f(x)\bigr)dx\leq\int_{-\pi}^\pi \phi\bigl(u_{f_H}(x)\bigr)dx\leq\int_{-\pi}^\pi \phi\bigl(u_{f^\#}(x)\bigr)dx.\]
Additionally, we obtain the inequality 
\[||u_f||_{L^\infty}\leq||u_{f^\#}||_{L^\infty},\]
not as just a consequence of Theorem \(\ref{sdrconvex}\), but also in a more straightforward way (see Theorem \ref{max of uf}), and we provide a condition on when it holds as an equality (something that cannot be deduced by the Theorem \(\ref{sdrconvex}\) itself).
\vspace{3mm}
\par The paper is structured as follows: In Section 2, we present essential properties of \(u_f\), we provide the precise definition of s.d.r., some properties of s.d.r and rearrangements in general. Section 3 is dedicated to the polarization. In Section 4, we have two Hardy-Littlewood inequalities and the previously mentioned Theorem \(\ref{max of uf}\) about the \(L^\infty\)-norms. Finally, the proofs of our main theorems are given in Section 5.
\section{Preliminaries}
\subsection{The Robin problem for the Poisson equation} For the Robin problem \((\ref{Poisson})\), \((\ref{Robin})\), the following existence and uniqueness property holds.
\begin{proposition}\rm{\cite[\text{Proposition 2.1}]{Langford}}\label{uf} 
Let \(f\in L^1[-\pi,\pi]\) and \(\alpha>0\). There exists a unique function \(u_f\in C^1[-\pi,\pi]\) that solves the Robin problem:
\begin{enumerate}
    \item[(a)] \(u'\) is absolutely continuous on \([-\pi,\pi]\),
    \item[(b)] \(-u''=f\) almost everywhere on \((-\pi,\pi),\)
    \item[(c)] \(-u'(-\pi)+\alpha u(-\pi)=u'(\pi)+\alpha u(\pi)=0.\)
\end{enumerate}
\end{proposition}
For \(\alpha>0\), the Green's function for the Robin problem \((\ref{Poisson})\), \((\ref{Robin})\) is
\[G(x,y)\coloneqq-\frac{1}{2}c_\alpha xy-\frac{1}{2}|x-y|+\frac{1}{2c_\alpha},\quad x,y\in[-\pi,\pi],\]
where \(\displaystyle c_a\coloneqq\dfrac{\alpha}{1+\alpha\pi}\in\Bigl(0,\frac{1}{\pi}\Bigr).\)
\begin{proposition}\label{green}\rm{\cite[\text{Proposition \rm{2.3}}]{Langford}}
 For \(\alpha>0\), the solution \(u_f\), provided by Proposition \(\ref{uf}\), can be represented as
\begin{equation}\label{anaparastash}
u_f(x)=\int_{-\pi}^\pi G(x,y)f(y)dy.
\end{equation}
\end{proposition}
The proof of the next Proposition will be omitted. It can be shown using simple Calculus and the reader is invited to verify it.
\begin{proposition}\label{idiotitesuf}
Let \(\alpha>0\), \(c_1,c_2\in\mathbb{R}\) and \(f,f_1,f_2\in L^1[-\pi,\pi]\) such that \(f\geq0\) and \(f\) is non-zero on a set of positive Lebesgue measure. Let \(G\) be the Green's function for the problem \((\ref{Poisson})\), \((\ref{Robin})\). Finally, let \(x_1,x_2\in(-\pi,\pi)\) with \(x_1<x_2\). Then:
\begin{enumerate}
    \item[(a)] \(G(x,y)=G(y,x)>0\) for all \(x,y\in[-\pi,\pi]\).
    \item[(b)] \(u_f\) is a concave function that is strictly positive and attains minimum at \(x=-\pi\) or \(x=\pi\). The function \(u_f\) does not attain maximum on either end of the interval \([-\pi,\pi]\) and, if \(x_1,x_2\) are points of maximum of \(u_f\) with \(x_1<x_2\), then \(u_f\) is constant on the interval \([x_1,x_2].\)
    \item[(c)] \(u_{c_1 f_1+c_2 f_2}=c_1 u_{f_1}+c_2 u_{f_2}.\)
\end{enumerate}
\end{proposition}
\subsection{Symmetric decreasing rearrangement} Now, we shall give the definition of the decreasing rearrangement and subsequently the symmetric decreasing rearrangement mentioned in the introduction. See \cite{Baernstein} and \cite{rearrangement} for more details and proofs.
\begin{definition}[Decreasing and Symmetric Decreasing Rearrangements in one dimension]\label{sdr} Let \(X\) be a Lebesgue measurable subset of the real line and \(f\in L^1(X)\) such that the Lebesgue measure of the set \(\{x\in X: f(x)>t\}\) is finite for all \(\displaystyle t>\essinf_X f\). For \(\lambda(A)\) be the Lebesgue measure of a set \(A\). We define \(f^\ast:[0,\lambda(X)]\to\overline{\mathbb{R}}\) as 
\[f^\ast(t)=
\begin{dcases}
\esssup_X f, & t=0,\\
\inf\bigl\{s\in\mathbb{R}:\lambda\bigl(\{x\in X:f(x)>s\}\bigr)\leq t\bigr\}, &t\in(0,\lambda(X)),\\
\essinf_X f, &t=\lambda(X).
\end{dcases}\]
The function \(f^\ast\) is called the decreasing rearrangement of \(f\).\\
We also define \(f^\#:[-\frac{1}{2}\lambda(X),\frac{1}{2}\lambda(X)]\to\overline{\mathbb{R}}\) the symmetric decreasing rearrangement of \(f\), given by \(f^\#(t)=f^\ast(2|t|).\)
\end{definition}
Without getting into much detail, we will call two integrable functions \(f:X\to\overline{\mathbb{R}}\) and \(g:Y\to\overline{\mathbb{R}}\) such that the sets \(\{x\in X: f(x)>t\}\) and \(\{y\in Y: g(y)>s\}\) have finite Lebesgue measure for all \(\displaystyle t>\essinf_X f\) and for all \(\displaystyle s>\essinf_Y g\) respectively, {\it equidistributed (with respect to the Lebesgue measure)} or {\it rearrangements of each other} if \(\lambda(X)=\lambda(Y)\) and \(f^\ast=g^\ast\) on \([0,\lambda(X)]\) (and \(f^\#=g^\#\) on \([-\frac{1}{2}\lambda(X),\frac{1}{2}\lambda(X)]\)).\\
\par 
Let us present some of the properties of the decreasing and the symmetric decreasing rearrangements. The following proposition combines results found throughout the first chapter of \cite{Baernstein}, with emphasis on Propositions 1.7 and 1.30.
\begin{proposition}\label{idiotitessdr}
Let \(X\subseteq\mathbb{R}\) be a Lebesgue measurable set and \(f:X\to\overline{\mathbb{R}}\) be an integrable function such that \(\lambda\bigl(\{x\in X:f(x)>t\}\bigr)<\infty\) for all \(\displaystyle t>\essinf_X f\). Then:
\begin{enumerate}
    \item[(a)] \(f\), \(f^\ast\) and \(f^\#\) are all rearrangements of each other.
    \item[(b)] \(f^\ast\) is non-increasing in \([0,\lambda(X)]\), right-continuous on \([0,\lambda(X))\) and also \(f^\ast(\lambda(X))=\displaystyle\lim_{t\to\lambda(X)}f^\ast(t).\)
    \item[(c)] \(f^\#\) is lower semi-continuous and even function on \([-\frac{1}{2}\lambda(X),\frac{1}{2}\lambda(X)]\).
    \item[(d)]\(f^\#(y)\leq f^\#(x)\) for \(0\leq|x|\leq|y|\).
    \item[(e)] The limit \(\displaystyle\lim_{t\to\frac{1}{2}\lambda(X)}f^\#(t)\) is equal to \(f^\#(\frac{1}{2}\lambda(X))\) and \(f^\#(0)=f^\ast(0)\).
\end{enumerate}
\end{proposition}
\begin{proposition}\label{greensdr}
    Let \(\alpha>0\), \(G\) the Green's function, as described in Proposition \(\mathrm{\ref{green}}\), and let \(x_0\in[-\pi,\pi]\). Then
    \begin{equation}\label{sdr of green}
    G^\#(x_0,y)\leq G(0,y)\text{ for all }y\in[-\pi,\pi].
    \end{equation}
    Equality holds for all \(y\in[-\pi,\pi]\) if and only if equality holds for one \(y\in[-\pi,\pi]\) if and only if \(x_0=0.\)\\
    Consider a function \(0\leq f\in L^1[-\pi,\pi]\) that is not a.e. equal to 0 and let \(\lambda_f\in(0,2\pi]\) be the Lebesgue measure of the set \(\{y\in[-\pi,\pi]:f(y)\neq0\}\). Then \(f^\#\) is strictly positive on \(\bigl(-\frac{\lambda_f}{2},\frac{\lambda_f}{2}\bigr)\) and, subsequently, 
    \begin{equation}\label{sdr epi f}
    0<G^\#(x_0,y)f^\#(y)\leq G(0,y)f^\#(y)\text{ for all }y\in\biggl(-\frac{\lambda_f}{2},\frac{\lambda_f}{2}\biggr),
    \end{equation}
    with equality holding for all \(y\in\bigl(-\frac{\lambda_f}{2},\frac{\lambda_f}{2}\bigr)\) if and only if it holds for one \(y\in\bigl(-\frac{\lambda_f}{2},\frac{\lambda_f}{2}\bigr)\) if and only if \(x_0=0.\)
\end{proposition}
\begin{proof}
For \(x_0=0\), it is quite obvious that \(G(0,y)=-\dfrac{1}{2}|y|+\dfrac{1}{2c_\alpha}\) is its own symmetric decreasing rearrangement and that the inequalities \((\ref{sdr of green})\) and \((\ref{sdr epi f})\) hold trivially as equalities.\\
Let \(x_0\in(0,\pi]\). Using Calculus, one can verify that \(G(x_0,\cdot)\) attains maximum at \(y=x_0\) and minimum at \(y=-\pi\), and that the following inequality holds
\begin{equation}\label{polygonikh green}
0<G(x_0,-\pi)<G(0,\mypm\pi)<G(x_0,\pi)\leq G(x_0,x_0)<G(0,0),
\end{equation}
and if \(x_0\neq\pi,\) then the inequality is strict.
The function \(G(x_0,\cdot)\) is linearly increasing on the interval \([-\pi,x_0]\), and linearly decreasing on \([x_0,\pi]\) (the graph of \(G(\pi,\cdot)\) is just an increasing line). Consider
\[y_0=\dfrac{c_ax_0\pi+\pi-2x_0}{c_ax_0-1}\in(-\pi,\pi].\]
It is easy to see that \(y_0\) solves the equation  \(G(x_0,y)=G(x_0,\pi).\) 
The s.d.r. of \(G(x_0,\cdot)\) is, by definition, an even function, so it suffices to study it on \([-\pi,0]\). We have that \(G^\#(x_0,\cdot)\) is linearly increasing on the intervals \([-\pi,y_1]\) and \([y_1,0],\) with \(G^\#(x_0,-\pi)=G(x_0,-\pi),\) \(G^\#(x_0,0)=G(x_0,x_0)\) and 
\(G^\#(x_0,y_1)=G(x_0,y_0)=G(x_0,\pi)\), where \(y_1=\dfrac{y_0-\pi}{2}\in(-\pi,0]\). From the pointwise linearity of \(G^\#(x_0,\cdot)\) and inequality \((\ref{polygonikh green})\), to prove \((\ref{greensdr})\) as a strict inequality, it suffices to show that \(G^\#(x_0,y_1)<G(0,y_1)\). This last inequality is equivalent to \(c_a\pi<1\), which is true. The inequality \((\ref{sdr epi f})\) is a direct consequence of \((\ref{sdr of green}).\)
The proof is completely analogous if \(x_0\in[-\pi,0)\).
\end{proof}
\begin{figure}[htp]\centering
\begin{tabular}{@{}cc@{}}
\begin{tikzpicture}[define rgb/.code={\definecolor{mycolor}{RGB}{#1}},
                    rgb color/.style={define rgb={#1},mycolor}]
\begin{axis}
[  
    x=8mm,
    y=8mm,
    xtick={-3,-2,-1,0,1,2,3},
    xmin=-3.5,
    xmax=3.5,
    xlabel={\tiny $y$},
    extra x ticks = {0},
    axis x line=middle,
    ytick={0,1,2,3},
    tick label style={font=\tiny},
    ymin=0,
    ymax=3.5,
    axis y line=middle,
    no markers,
    samples=100,
    domain=-pi:pi,
    restrict y to domain=-pi:pi,
]
\addplot[blue] {x/6-(abs(x+(2*pi/3)))/2+pi}; 
\addplot[black] coordinates { (-pi,pi/2) (0,pi) (pi,pi/2) };
\addplot[red] coordinates { (-pi,3*pi/8) (pi/2,15*pi/16) (pi,5*pi/8) };
\addplot[mygreen1] {x/4+pi/2};
\end{axis}
\end{tikzpicture}
&\begin{tikzpicture}[define rgb/.code={\definecolor{mycolor}{RGB}{#1}},
                    rgb color/.style={define rgb={#1},mycolor}]
\begin{axis}
[  
    x=8mm,
    y=8mm,
    xtick={-3,...,3},   
    xmin=-3.5,
    xmax=3.5,
    xlabel={\tiny $y$},
    extra x ticks = {0},
    axis x line=middle,
    ytick={0,1,2,3},
    tick label style={font=\tiny},
    ymin=0,
    ymax=3.5,
    axis y line=middle,
    no markers,
    samples=100,
    domain=-pi:pi,
    restrict y to domain=-pi:pi
]
\addplot[blue] coordinates { (-pi,pi/3) (-pi/2,2*pi/3) (0,8*pi/9) (pi/2,2*pi/3) (pi,pi/3) };
\addplot[black] coordinates { (-pi,pi/2) (0,pi) (pi,pi/2) };
\addplot[red] coordinates { (-pi,3*pi/8) (-2*pi/3,5*pi/8) (0,15*pi/16) (2*pi/3,5*pi/8) (pi,3*pi/8) };
\addplot[mygreen1] coordinates { (-pi,pi/4) (0,3*pi/4) (pi,pi/4) };
\end{axis}
\end{tikzpicture}
\\ (a) & (b) \end{tabular}
\caption{(a) Graphs of \(G(x_0,y)\) for \(\alpha=1/\pi\) and various values of \(x_0\) \((x_0=\color{blue}{-2\pi/3},\color{black}{0},\color{red}{\pi/2},\color{mygreen1}{\pi}\color{black}{)}\); (b) Graphs of their symmetric decreasing rearrangements.}
\end{figure}
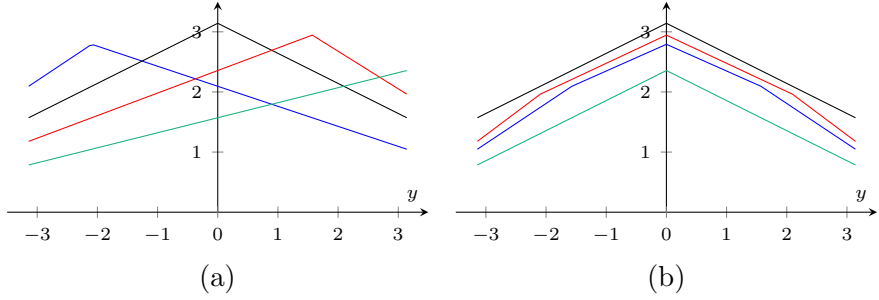
\section{Polarization}
\begin{definition}\label{polar} Let \(b\in(-\pi,0)\cup(0,\pi)\) and \(H\) be the open half-line of \(\mathbb{R}\setminus\{b\}\) which contains \(0\). 
Let \(f:\mathbb{R}\to\overline{\mathbb{R}}\) be a Lebesgue integrable function such that the Lebesgue measure of the set \[\{x\in\mathbb{R}:f(x)>t\}\] is finite for all \(t>\essinf f.\)
\begin{itemize}
    \item If \(b\in(0,\pi),\) we define the polarization (towards zero) of \(f\) with respect to \(b\) (or the polarization of \(f\) with respect to \(H=(-\infty,b)\)) as
    \[f_H(x)=
    \begin{dcases}
        f(x), &x\in[-\pi,2b-\pi),\\
        \max\{f(x),f(2b-x)\}, &x\in[2b-\pi,b],\\
        \min\{f(x),f(2b-x)\}, &x\in[b,\pi].
    \end{dcases}\]
    \item If \(b\in(-\pi,0),\) we define the polarization (towards zero) of \(f\) with respect to \(b\) (or the polarization of \(f\) with respect to \(H=(b,+\infty)\)) as
    \[f_H(x)=
    \begin{dcases}
        \min\{f(x),f(2b-x)\}, &x\in[-\pi,b],\\
        \max\{f(x),f(2b-x)\}, &x\in[b,2b+\pi],\\
        f(x), &x\in(2b+\pi,\pi].
    \end{dcases}\]
\end{itemize}
\end{definition}
\begin{remark}
    \(f_H\) is a rearrangement of \(f\) and thus \((f_H)^\ast=f^\ast\) and \((f_H)^\#=f^\#\).
\end{remark}
\begin{lemma}\rm{\cite[\text{Lemma \rm{1.38}}]{Langford}}\label{polar of non-sdr} Let \(0\leq f\in L^1[-\pi,\pi]\) be a function which is not a.e. equal to its symmetric decreasing rearrangement \(f^\#\). Then there exists an open half-line \(H\) that contains 0 such that \(f_H\) is not a.e. equal to \(f\).
\end{lemma}
By Proposition \ref{greensdr}, the function \(G_{x_0}\coloneqq G(x_0,\cdot)\) is its own symmetric decreasing rearrangement if and only if \(x_0=0\). So, we have the following example related to the previous Lemma.
\begin{example}\label{paradeigma}
Fix a point \(x_0\in[-\pi,0)\cup(0,\pi]\) and consider the Green's function \(G_{x_0}=G(x_0,\cdot):[-\pi,\pi]\to(0,+\infty)\) with one argument being constant. Let \(h\) be the finite end of an open half-line \(H\) such that \(0\in H\). Let \((G_{x_0})_H\) be the polarization of \(G_{x_0}\) with respect to \(H\).
\begin{itemize}
    \item If \(x_0<0\), then \((G_{x_0})_H\equiv G_{x_0}\) if and only if \(h\in(-\infty,x_0]\cup(0,+\infty)\).
    \item If \(x_0>0\), then \((G_{x_0})_H\equiv G_{x_0}\) if and only if \(h\in(-\infty,0)\cup[x_0,+\infty)\).
\end{itemize}
\end{example}
\section{Integral inequalities}
The following two Propositions are direct consequences of the Theorems 2.9 and 2.15 of \cite{Baernstein}, respectively. We apply the Theorems for \(\Psi(x,y)=xy\) and \(n=1\).
\begin{proposition}\label{integral of polar}Let \(0\leq f,g\in L^1(\mathbb{R})\). Then
\[\int_{-\infty}^\infty f(y)g(y)dy\leq\int_{-\infty}^\infty f_H(y)g_H(y)dy.\]
Equality holds if and only if the set
\[A_H=B_H\cup C_H\text{ has zero Lebesgue measure,}\]
where
\[B_H=\{y\in H:f(y)<f(\overline{y})\text{ and }g(y)>g(\overline{y})\}\]
and
\[C_H=\{y\in H:f(y)>f(\overline{y})\text{ and }g(y)<g(\overline{y})\}\]
\end{proposition}
\begin{proposition}\label{integral of sdr}Let \(0\leq f,g\in L^1(\mathbb{R})\). Then
\[\int_{-\infty}^\infty f(y)g(y)dy\leq\int_{-\infty}^\infty f^\#(y)g^\#(y)dy.\]
Equality holds if and only if the set
\[A=\bigl\{(x,y)\in\mathbb{R}^2:f(x)<f(y)\text{ and }g(x)>g(y)\bigr\}\]
has zero two-dimensional Lebesgue measure.
\end{proposition}
\begin{remark} We would like to use the previous two theorems for \(0\leq f\in L^1[-\pi,\pi]\) and \(g=G_{x_0}\) (as described in Example \ref{paradeigma}), which are defined on the interval \([-\pi,\pi]\). In order to be rigorous, we have to extend them on the whole real line as \(\widehat{f}\) such that \(\widehat{f}=f\) on \([-\pi,\pi]\) and \(\widehat{f}=0\) elsewhere, and similarly for \(g\). That being said, we won't bother with the ``\(\widehat{\phantom{x}}\)'' notation ever again.
\end{remark}
\begin{theorem}\label{max of uf}Let \(\alpha>0\), \(0\leq f\in L^1[-\pi,\pi]\) and \(u_f\) be the solution to the corresponding Robin problem. Let \(f^\#\) be the symmetric decreasing rearrangement of \(f\) and \(u_{f^\#}\) the solution to the Robin problem with heat source the function \(f^\#\). Then 
\begin{equation}\label{ineq max}
||u_f||_{L^\infty}\leq||u_{f^\#}||_{L^\infty}=u_{f^\#}(0).
\end{equation}
Equality if and only if \(f=f^\#\) a.e..
\end{theorem}
\begin{proof}
Let \(x_0\) be the point of maximum of \(u_f\) that has the least absolute value amongst all points of maximum of \(u_f\). We extend the positive continuous function \(G_{x_0}\coloneqq G(x_0,\cdot):[-\pi,\pi]\to\mathbb{R}\) and the non-negative function \(f\) to be identically zero on the exterior of \([-\pi,\pi]\). The desired inequality is derived as follows: 
\begin{align*}
    ||u_f||_\infty&=u_f(x_0)=\int_{-\pi}^\pi G(x_0,y)f(y)dy=\int_{-\infty}^\infty G_{x_0}(y)f(y)dy\\
    &\overset{(1)}{\leq}\int_{-\infty}^\infty (G_{x_0})_H(y)f_H(y)dy\overset{(2)}{\leq}
    \int_{-\infty}^\infty \bigl((G_{x_0})_H\bigr)^\#(y)(f_H)^\#(y)dy\\
    &=\int_{-\infty}^\infty (G_{x_0})^\#(y)f^\#(y)dy=\int_{-\pi}^\pi (G_{x_0})^\#(y)f^\#(y)dy\\
    &\overset{(3)}{\leq}
    \int_{-\pi}^\pi G(0,y)f^\#(y)dy=u_{f^\#}(0)=||u_{f^\#}||_\infty,
\end{align*}
where \(H\) is any open half-line in \(\mathbb{R}\). The inequalities (1), (2) and (3) are justified by Propositions \ref{integral of polar}, \ref{integral of sdr} and \ref{greensdr} respectively.\\
Now, quite trivially, if \(f=f^\#\) almost everywhere, then \(u_f\equiv u_{f^\#}\) and thus \(||u_f||_\infty=||u_{f^\#}||_\infty\).\\
Conversely, assume that \(f\) is not almost everywhere equal to \(f^\#\). We will show that \((\ref{ineq max})\) cannot hold as an equality. In order for \((\ref{ineq max})\) to hold as an equality, all three (1), (2) and (3) must hold as equalities, too. In particular, it suffices to concentrate first on (3) and then on (1).\\
If \(x_0\neq0\), then, by Proposition \ref{greensdr}, (3) holds as a strict inequality. So, let us assume that \(x_0=0\), but still \(f\) is not almost everywhere equal to its symmetric decreasing rearrangement.
By Lemma \ref{polar of non-sdr}, there exists an open half-line \(H\) that contains 0 such that \(f_H\) is not a.e. equal to \(f\). Let us compute the set \(A_H\) mentioned
in Proposition \ref{integral of polar} for \(g(y)=G(0,y)\). Since the function \(G(0,y)=-\dfrac{1}{2}|y|+\dfrac{1}{2c_\alpha}\) is its own symmetric decreasing rearrangement, it is easy to see that \(G(0,y)>G(0,\overline{y})\) for all \(y\in H\). Thus
\begin{align*}
B_H&=\{y\in H:f(y)<f(\overline{y})\text{ and }G(0,y)>G(0,\overline{y})\}\\
&=\{y\in H:f(y)<f(\overline{y})\}
\end{align*}
and
\[C_H=\{y\in H:f(y)>f(\overline{y})\text{ and }G(0,y)<G(0,\overline{y})\}
=\varnothing,\]
so \(A_H=\{y\in H:f(y)<f(\overline{y})\}\).
So, the condition that \(\lambda(A_H)=0\) turns out to be equivalent to \(f(y)\geq f(\overline{y})\) almost everywhere on \(H\) which, in turn, is equivalent to \(f\) being equal to \(f_H\) almost everywhere on \(H.\) But, as mentioned above, \(f\) is not a.e. equal to \(f_H\), thus \(\lambda(A_H)>0\) and, by Proposition \ref{integral of polar}, (1), and subsequently \((\ref{ineq max})\), holds as a strict inequality.
\end{proof}
\section{Proofs of main Theorems}
The following Lemma incorporates Theorem 1 and Note 2 of \cite{Karamata}.
\begin{lemma}[Karamata's inequality]\label{karamata}
Let \(\phi:(a,b)\to\mathbb{R}\) be a convex and increasing function. Let \(\{x_i\}_{i=1}^n\) and \(\{y_i\}_{i=1}^n\) be two n-tuples of real numbers contained in the open interval \((a,b)\) such that
\begin{enumerate}
    \item[(a)] \(x_1\geq x_2\geq\dots\geq x_n\) and \(y_1\geq y_2\geq\dots\geq y_n\),
    \item[(b)] \(x_1\leq y_1\),
    \item[(c)] \(x_1+x_2+\dots+x_k\leq y_1+y_2+\dots+y_k\) for all \(k\in\{2,\dots,n\}\).
\end{enumerate}
Under these assumptions, the following inequality holds
\begin{equation}\label{karamata ineq}
\phi(x_1)+\phi(x_2)+\dots+\phi(x_n)\leq\phi(y_1)+\phi(y_2)+\dots+\phi(y_n).
\end{equation}
If the function \(\phi\) is strictly convex, then \((\ref{karamata ineq})\) holds as an equality if and only if \(x_i=y_i\) for all \(i=1,2,\dots,n.\)
\end{lemma}
Now, we shall present some more properties of Green's functions, which, in turn, will imply some useful inequalities involving the solutions of Robin's problems.
\begin{lemma}\label{karamata green} Let \(\alpha>0\), let \(G\) be the Green's function for the problem \((\ref{Poisson})\), \((\ref{Robin})\), let \(b\in(-\pi,0)\cup(0,\pi)\). If \(b>0\) (respectively \(b<0\)), let \(I\) be the interval \([b,\pi]\) (respectively \([-\pi,b]\)). For all \(x,y\in I\), the following two inequalities hold
\begin{equation}\label{5.2}
G(x,y)+G(x',y)\leq G(x,y')+G(x',y')
\end{equation}
and
\begin{equation}\label{5.3}
G(x,y)+G(x,y')\leq G(x',y)+G(x',y'),
\end{equation}
where \(x'\) and \(y'\) are the reflections of \(x\) and \(y\) with respect to \(b\), i.e. \(x'=2b-x\) and \(y'=2b-y\).
\begin{proof}
Due to symmetry, we can consider only the case that \(b\in(-\pi,0)\). Therefore, we have \(I=[-\pi,b]\). Let \(x,y\in I\) and write down the formulas for \(G(x,y)\), \(G(x',y)\), \(G(x,y')\) and \(G(x',y')\). Then, it is obvious that \((\ref{5.2})\) is equivalent to
\begin{equation}\label{5.4}
c_\alpha xy+|x-y|+c_\alpha x'y+|x'-y|\geq c_\alpha xy'+|x-y'|+c_\alpha x'y'+|x'-y'| .
\end{equation}
Note that \(|x-y|=|x'-y'|\) and \(|x'-y|=|x-y'|\). Also, the constant \(c_a\) is positive, which ultimately means that the inequality we have to prove is equivalent to
\begin{equation}\label{5.5}
xy+x'y\geq xy'+x'y'.
\end{equation}
Substituting \(x'=2b-x\) and \(y'=2b-y\) we see that \((\ref{5.5})\) is equivalent to \(y\leq b\), which is true because \(y\in I\).
In the exact same manner, one can prove \((\ref{5.3})\).
\end{proof}
\end{lemma}
This lemma has immediate consequences on the solution of a Robin problem with given heat source \(f\) and the solution of the Robin problem with heat source any polarization \(f_H\) of \(f\).
\begin{lemma}\label{karamata uf}
Let \(\alpha>0\), \(0\leq f\in L^1[-\pi,\pi]\) and \(u_f\) be the solution to the corresponding Robin problem, let \(b\in(-\pi,0)\cup(0,\pi)\), let \(f_H\) be the polarization of \(f\) w.r.t. \(b\), and let \(u_{f_H}\) be the solution to the Robin problem with heat source \(f_H\). Let \(I\) be as defined in Lemma \(\mathrm{\ref{karamata green}}\). The following hold:
\begin{enumerate}
    \item[(a)] If \(x\in I\), let \(x'=2b-x\). Then
    \begin{equation}\label{uf syn uf}
    u_f(x)+u_f(x')\leq u_{f_H}(x)+u_{f_H}(x').
    \end{equation}
    \item[(b)] If \(x\in I\), let \(x'=2b-x\). Then
    \begin{equation}\label{uf uf polwsh polwsh}
    u_f(x)\leq u_{f_H}(x').
    \end{equation}
    \item[(c)] If \(b>0\) (respectively \(b<0\)), let \(x\in[-\pi,b]\) (respectively \(x\in [b,\pi]\)). Then 
    \begin{equation}\label{uf uf polwsh}
    u_f(x)\leq u_{f_H}(x).
    \end{equation}
    For \(b>0\) (resp. \(b<0\)), \((\ref{uf uf polwsh})\) holds as an equality for some \(x\in[-\pi,b]\) (resp. \(x\in [b,\pi]\)) if and only if it holds as an equality for all \(x\in[-\pi,b]\) (resp. \(x\in [b,\pi]\)) if and only if \(f=f_H\) almost everywhere on \([-\pi,\pi]\) if and only if \(u_f\equiv u_{f_H}\) on \([-\pi,\pi]\). 
\end{enumerate}
\end{lemma}
\begin{proof} 
Due to symmetry, we can consider only the case that \(b\in(-\pi,0)\).\\
 (a) Let \(x\in[-\pi,b]\) and \(x'=2b-x\in[b,2b+\pi].\) From \((\ref{5.2})\) we have that \(G(x,\cdot)+G(x',\cdot)\) coincides with its polarization with respect to \(b\). Thus, by Proposition \ref{integral of polar},
    \begin{align*}
    u_f(x)+u_f(x')&=\int_{-\pi}^\pi\bigl(G(x,y)+G(x',y)\bigr)f(y)dy\\
    &\leq\int_{-\pi}^\pi\bigl(G(x,y)+G(x',y)\bigr)_H f_H(y)dy\\
    &=\int_{-\pi}^\pi\bigl(G(x,y)+G(x',y)\bigr)f_H(y)dy=
    u_{f_H}(x)+u_{f_H}(x').
    \end{align*}
 (b) Let \(x\in[-\pi,b]\) and \(x'=2b-x\in[b,2b+\pi].\) By the definition of polarization, \(f\equiv f_H\) on \((2b+\pi,\pi]\). 
    It is quite elementary to show that \(G(x,y)\leq G(x',y)\) for all \(y\in[2b+\pi,\pi]\), and thus\\
    \[\int_{2b+\pi}^\pi G(x,y)f(y)dy\leq\int_{2b+\pi}^\pi G(x',y)f(y)dy=\int_{2b+\pi}^\pi G(x',y)f_H(y)dy.\]
    Directly by the definition of polarization, \(f_H(y)\leq f_H(y')\) for all \(y\in[-\pi,b],\) \(y'=2b-y\) and either
    \[f_H(y)=f(y)\text{ and } f_H(y')=f(y')\]
    or
    \[f_H(y)=f(y')\text{ and } f_H(y')=f(y).\]
    Combining these facts with Lemma \ref{karamata green} and integrating over \([-\pi,b]\) with respect to \(y\) yields 
    \[\int_{-\pi}^b\bigl(G(x,y)f(y)+G(x,y')f(y')\bigr)dy\leq\!\!\int_{-\pi}^b\bigl(G(x',y)f_H(y)+G(x',y')f_H(y')\bigr)dy,\]
    which in turn, after a change of variables, gives
    \[\int_{-\pi}^{2b+\pi} G(x,y)f(y)dy\leq\int_{-\pi}^{2b+\pi} G(x',y)f_H(y)dy.\]
    We have already shown that
    \[\int_{2b+\pi}^\pi G(x,y)f(y)dy\leq\int_{2b+\pi}^\pi G(x',y)f_H(y)dy,\]
    and, by adding the two inequalities together, we get that \(u_f(x)\leq u_{f_H}(x').\)\\
 (c) Let \(x\in[b,\pi]\). By Example \ref{paradeigma}, \(G(x,\cdot)\equiv \bigl(G(x,\cdot)\bigr)_H\). Thus, by Proposition \ref{integral of polar},
\begin{equation}\label{apodeiksh karamata uf}
\begin{aligned}
    u_f(x)&=\int_{-\pi}^\pi G(x,y)f(y)dy\leq\int_{-\pi}^\pi \bigl(G(x,y)\bigr)_H f_H(y)dy\\
    &=\int_{-\pi}^\pi G(x,y)f_H(y)dy=u_{f_H}(x).
\end{aligned}
\end{equation}
    We examine the case of \((\ref{apodeiksh karamata uf})\) holding as an equality in accordance with the condition for equality on Proposition \ref{integral of polar}, namely we compute the corresponding set \(A_H\) in our context and determine when its Lebesgue measure is zero. We do not have to show all the equivalences on Lemma \ref{karamata uf}.(c) since most of them have already been established to be true. It suffices to show that if \((\ref{apodeiksh karamata uf})\) holds as an equality for one \(x\in[b,\pi]\), then \(f=f_H\) a.e. on \([-\pi,\pi]\). After some simplifications, namely that \(f\) and \(G(x,\cdot)\) are identically zero outside of \([-\pi,\pi]\) and that \(G(x,\cdot)\) is its own polarization with respect to \(b\), we deduce that \(A_H=\{y\in(b,2b+\pi]: f(y)<f(y')\}\), which has zero Lebesgue measure if and only if \(f(y)\geq f(y')\) almost everywhere on \((b,2b+\pi]\), i.e. if and only if \(f\) and \(f_H\) coincide almost everywhere on \([-\pi,\pi]\).
\end{proof}
We now apply the Karamata inequality (Lemma \ref{karamata}) for \(n=2\) on values of solutions to the Robin problem.
\begin{proposition}\label{karamata phi}
Let \(\alpha>0\), \(0\leq f\in L^1[-\pi,\pi]\) and \(u_f\) be the solution to the corresponding Robin problem. Let \(b\in(-\pi,0)\cup(0,\pi)\), let \(I\) be as defined in Lemma \(\mathrm{\ref{karamata green}}\). Let \(f_H\) be the polarization of \(f\) w.r.t. \(b\), and let \(u_{f_H}\) be the solution to the Robin problem with heat source \(f_H\). Let \(\phi:\mathbb{R}\to\mathbb{R}\) be a convex and increasing function. Then
\begin{enumerate}
    \item[(a)] If \(x\in I\), let \(x'=2b-x\). Then
    \begin{equation}\label{karamata ineq phi}
    \phi\bigl(u_f(x)\bigr)+\phi\bigl(u_f(x')\bigr)\leq \phi\bigl(u_{f_H}(x)\bigr)+\phi\bigl(u_{f_H}(x')\bigr).
    \end{equation}
    \item[(b)] If \(b>0\) (respectively \(b<0\)), let \(x\in[-\pi,b]\) (respectively \(x\in [b,\pi]\)). Then 
    \begin{equation}\label{monotonia phi}
    \phi\bigl(u_f(x)\bigr)\leq \phi\bigl(u_{f_H}(x)\bigr).
    \end{equation}
    If \(\phi\) is strictly increasing and \(b>0\) (resp. \(b<0\)), then \((\ref{monotonia phi})\) holds as an equality for some \(x\in[-\pi,b]\) (resp. \(x\in [b,\pi]\)) if and only if it holds as an equality for all \(x\in[-\pi,b]\) (resp. \(x\in [b,\pi]\)) if and only if \(f=f_H\) almost everywhere on \([-\pi,\pi]\).
\end{enumerate}
\end{proposition}
\begin{proof}
(a) Use Lemma \ref{karamata uf}.(a) to apply the Karamata inequality for \(n=2\), \(\{x_1,x_2\}=\{u_f(x),u_f(x')\}\) and \(\{y_1,y_2\}=\{u_{f_H}(x),u_{f_H}(x')\}\). By Proposition \ref{idiotitesuf} and Theorem \ref{max of uf}, all values of \(u_f\), \(u_{f_H}\) and \(u_{f^\#}\) belong to the interval \((0,u_{f^\#}(0)]\). So, we use Lemma \ref{karamata} and we get the desired inequality for all functions \(\phi\) that are convex on intervals \((0,M)\), for any real number \(M>u_{f^\#}(0)\). But, quite trivially, this class of functions contains the real functions that are convex on the whole real line.\\
(b) The inequality \((\ref{monotonia phi})\) is a direct consequence  of Lemma \ref{karamata uf}.(c) and the monotonicity of \(\phi\). Assuming that \(\phi\) is strictly increasing we have equality here if and only if equality holds on \((\ref{uf uf polwsh})\).
\end{proof}
\justify
\textbf{Proof of Theorem \ref{polarconvex}.}
\begin{proof}
Due to symmetry, we may assume that \(b>0\).\\
Integrating both sides of the inequality \((\ref{karamata ineq phi})\) from \(b\) to \(\pi\), and with a change of variables, we immediately get that 
\begin{equation}\label{apodeiksi 1.1 prwth}
\int_{2b-\pi}^\pi \phi\bigl(u_f(x)\bigr)dx\leq\int_{2b-\pi}^\pi \phi\bigl(u_{f_H}(x)\bigr)dx.
\end{equation}
Even more simply, by \((\ref{monotonia phi})\), we conclude that
\begin{equation}\label{apodeiksi 1.1 deuterh}
\int_{-\pi}^{2b-\pi} \phi\bigl(u_f(x)\bigr)dx\leq\int_{-\pi}^{2b-\pi} \phi\bigl(u_{f_H}(x)\bigr)dx.
\end{equation}
Adding the inequalities \((\ref{apodeiksi 1.1 prwth})\) and \((\ref{apodeiksi 1.1 deuterh})\) we get \((\ref{ineq polar})\).\\
Assume that \(\phi\) is strictly increasing. In order for \((\ref{ineq polar})\) to hold as an equality, \((\ref{apodeiksi 1.1 deuterh})\) has to hold as an equality. If \(f\equiv f_H\), then clearly all \((\ref{apodeiksi 1.1 prwth})\), \((\ref{apodeiksi 1.1 deuterh})\) and \((\ref{ineq polar})\) hold as equalities.\\
Conversely, assume that \(f\) is not identical to \(f_H\) a.e. on \([-\pi,\pi].\) Then, by Proposition \ref{karamata phi}.(b), \(\phi\bigl(u_f(x)\bigr)< \phi\bigl(u_{f_H}(x)\bigr)\) for all \(x\in[-\pi,2b-\pi]\), and thus the inequality \((\ref{apodeiksi 1.1 deuterh})\) is also strict, rendering the inequality \((\ref{ineq polar})\) strict as well.
\end{proof}
The following lemma is a special case of Theorem 6.1 of \cite{Brock}.
\begin{lemma}\label{seq of polar to sdr}
Let \(0\leq f\in L^1[-\pi,\pi]\), let \(f^\#\) denote the symmetric decreasing rearrangement of \(f\). There exists a sequence of non-zero numbers \(\{b_n\}_{n=1}^\infty\subset(-\pi,\pi)\) such that
\[f_n\longrightarrow f^\#\quad\textit{ in }\:\:L^1[-\pi,\pi],\]
where \(f_1\) is the polarization of \(f\) with respect to \(b_1\) and \(f_{n+1}\) is the polarization of \(f_n\) w.r.t. \(b_{n+1}\), for \(n=1,2,\dots.\)
\vspace{0mm}\\
\noindent
Let \(u_{f_{n}}\) and \(u_{f^\#}\) be the solutions to the corresponding Robin problems for \(f_{n}\) and \(f^\#\) respectively. Then \(u_{f_{n}}\longrightarrow u_{f^\#}\) uniformly on \([-\pi,\pi]\).
\end{lemma}
\begin{proof}
In our context, the existence of a sequence \(\{f_n\}_{n=1}^\infty\) of iterated polarizations, as described in the statement of Lemma \ref{seq of polar to sdr}, such that \(f_n\longrightarrow f^\#\) in \(L^1[-\pi,\pi]\), is directly guaranteed by Theorem 6.1 of \cite{Brock}.\\
Finally, for all \(x\in[-\pi,\pi]\), we have
\begin{align*}
\bigl|u_{f_{n}}(x)-u_{f^\#}(x)\bigr|&=\left|\int_{-\pi}^\pi G(x,y)\bigl(f_n(y)-f^\#(y)\bigr)dy\right|\\
&\leq \max_{x,y\in[-\pi,\pi]}G(x,y)\int_{-\pi}^\pi |f_n(y)-f^\#(y)|dy\\
&=G(0,0)||f_n-f^\#||_{L^1}\xrightarrow{n\to\infty}0.
\end{align*}
\end{proof}
\justify
\textbf{Proof of Theorem \ref{sdrconvex}}
\begin{proof}
Let \(\{f_n\}_{n=1}^\infty\) be as in Lemma \ref{seq of polar to sdr}. Then, by Theorem \(\ref{polarconvex}\),
\begin{align*}
\int_{-\pi}^\pi \phi\bigl(u_f(x)\bigr)dx&\leq\int_{-\pi}^\pi \phi\bigl(u_{f_1}(x)\bigr)dx\leq\int_{-\pi}^\pi \phi\bigl(u_{f_2}(x)\bigr)dx\leq\dots\\
&\leq\int_{-\pi}^\pi \phi\bigl(u_{f_n}(x)\bigr)dx\leq\dots\:,
\end{align*}
meaning that the sequence of the integrals \(\displaystyle\int_{-\pi}^\pi \phi\bigl(u_{f_n}(x)\bigr)dx\) is increasing in \(n\).\\
By the uniform convergence of \(\{u_{f_n}\}_{n=1}^\infty\) to \(u_{f^\#}\) (Lemma \ref{seq of polar to sdr}),
\[\lim_{n\to\infty}\int_{-\pi}^\pi \phi\bigl(u_{f_n}(x)\bigr)dx=\int_{-\pi}^\pi \phi\bigl(u_{f^\#}(x)\bigr)dx,\]
which proves the inequality \((\ref{ineq sdr}).\)
If \(f=f^\#\), then \((\ref{ineq sdr})\) is a trivial equality.\\
Conversely, assume that \(\phi\) is strictly increasing and that \(f\) is not equal to its symmetric decreasing rearrangement a.e. on \([-\pi,\pi].\) Then, by Lemma \(\ref{polar of non-sdr}\), there exists some polarization (towards zero) \(f_H\) of \(f\) that is not a.e. equal to \(f\). This in turn, by Theorem \ref{polarconvex}, means that 
\[\int_{-\pi}^\pi \phi\bigl(u_f(x)\bigr)dx<\int_{-\pi}^\pi \phi\bigl(u_{f_H}(x)\bigr)dx.\]
Applying the first part of Theorem \ref{sdrconvex}, which we have already proven, to \(f_H\), we get that
\[\int_{-\pi}^\pi \phi\bigl(u_{f_H}(x)\bigr)dx\leq\int_{-\pi}^\pi \phi\bigl(u_{(f_H)^\#}(x)\bigr)dx.\]
But \(f\) and \(f_H\) are rearrangements of each other, so they have the same s.d.r. \(f^\#\), and thus
\[\int_{-\pi}^\pi \phi\bigl(u_f(x)\bigr)dx<\int_{-\pi}^\pi \phi\bigl(u_{f_H}(x)\bigr)dx\leq\int_{-\pi}^\pi \phi\bigl(u_{f^\#}(x)\bigr)dx.\]
\end{proof}
\section*{Acknowledgments}
I would like to thank Professor D. Betsakos, my advisor, for his advice during the preparation of this work.
\bibliography{bibliography}
\bibliographystyle{plain}
\nocite{*}
\end{document}